\documentclass[11pt]{amsart}

\usepackage{lineno}
\usepackage[margin=1.2in]{geometry}

 \usepackage{color}
 \usepackage{amssymb}
\usepackage[utf8]{inputenc}
 \usepackage{soul}

 \usepackage{graphicx}
\usepackage{bbm}
\usepackage[title]{appendix}
\usepackage{color}
\usepackage{tikz-cd}

 \usepackage{graphicx}
\usepackage{tikz-cd}

\newcommand{\bea}{\begin{eqnarray}}
\newcommand{\eea}{\end{eqnarray}}

\allowdisplaybreaks

\newcommand{\Z}{\Bbb Z}
\newcommand{\Ind}{\text{Ind}}
\newcommand{\C}{\Bbb C}

\newcommand{\g}{  \mathfrak g}

\newtheorem{theorem}{Theorem}[section]
\newtheorem{lemma}[theorem]{Lemma}

\newtheorem{main}{Main Theorem}
\theoremstyle{definition}

\theoremstyle{remark}
\newtheorem{remark}[theorem]{Remark}

\theoremstyle{proposition}
\newtheorem{proposition}[theorem]{Proposition}
\theoremstyle{corollary}
\newtheorem{corollary}[theorem]{Corollary}
\numberwithin{equation}{section}

\newcommand{\KL}{\mathrm{KL}}

\begin{document}



\title[]{ Tensor category $\KL_k(\mathfrak{sl}_{2n})$ via  minimal affine $W$-algebras at the non-admissible level $k =-\frac{2n+1}{2}$  }

 \author[]{Dra\v zen  Adamovi\' c}
\address{Department of Mathematics, Faculty of Science \\
	University of Zagreb \\
	Bijeni\v cka 30 }

\email{adamovic@math.hr}

\author[]{Thomas Creutzig}

 \address{ Department of Mathematical and Statistical Sciences, University of Alberta, 632 CAB,
Edmonton, Alberta, Canada T6G 2G1 }
 
  \email{creutzig@ualberta.ca}

\author[]{Ozren Per\v se}
\address{Department of Mathematics, Faculty of Science \\
	University of Zagreb \\
	Bijeni\v cka 30
}

\email{perse@math.hr}

\author[]{Ivana Vukorepa}
\address{Department of Mathematics, Faculty of Science \\
	University of Zagreb \\
	Bijeni\v cka 30
}
\email{vukorepa@math.hr} 

 \begin{abstract}
 We prove that   $\KL_k(\mathfrak{sl}_m)$ is a semi-simple, rigid braided tensor category for all even $m\ge 4$, and $k= -\frac{m+1}{2}$ which generalizes result from \cite{APV}   obtained for $m=4$. Moreover, all modules in $\KL_k(\mathfrak{sl}_m)$ are simple-currents and they appear in the decomposition of conformal embeddings  $\mathfrak{gl}_m \hookrightarrow \mathfrak{sl}_{m+1} $ at level $ k= - \frac{m+1}{2}$ from \cite{AKMPP-16}.  For this we inductively identify minimal affine $W$-algebra
$ W_{k-1} (\mathfrak{sl}_{m+2}, \theta)$ as simple current extension of $L_{k}(\mathfrak{sl}_m) \otimes \mathcal H \otimes \mathcal M$, where $\mathcal H$ is the rank one Heisenberg vertex algebra, and $\mathcal M$ the singlet vertex algebra for $c=-2$. The proof  uses previously obtained results  for the  tensor categories of singlet algebra from \cite{CMY}.  We also classify all irreducible ordinary modules for  $ W_{k-1} (\mathfrak{sl}_{m+2}, \theta)$. The semi-simple part of the category of $ W_{k-1} (\mathfrak{sl}_{m+2}, \theta)$--modules comes from   $\KL_{k-1}(\mathfrak{sl}_{m+2})$, using quantum Hamiltonian reduction, but this  $W$-algebra also contains indecomposable ordinary modules.
 \end{abstract}
\maketitle

\section{Introduction}

\subsection{ Representation theory of affine vertex algebras}
Let $L_k(\g)$ be the simple affine vertex algebra of level $k$ associated to the simple Lie algebra $\g$. 
Most vertex algebras (maybe even all of them) can be constructed from affine vertex algebras via standard constructions like tensor products, extensions, orbifolds, cosets and quantum Hamiltonian reduction. This means that affine vertex algebras are very central in the theory of vertex algebras and in particular 
the representation theory of $L_k(\g)$  is one of the most important directions. It connects to many fundamental concepts as tensor categories, quantum groups, modular forms, combinatorial identities and various connections with physics. 

Recall that $L_k(\g)$ is rational if and only if $k$ is a positive integer \cite{FZ}.  When $k$ is not a positive integer, then some suitable categories of $L_k(\g)$--modules could be semi-simple, but in general the category of all weak modules is non-semisimple.  In particular, when $k$ is an admissible rational number, which is not integral, then $L_k(\g)$ is semi-simple in the category $\mathcal O$ (cf. \cite{Ara}), but it also contains indecomposable and logarithmic modules (cf. \cite{AdM-2009,  A-2019, ACG-2021,CRR-21}). 

Denote by $h^\vee$ the dual Coxeter number of $\g$. Most rational levels, such that $k+h^\vee$ is positive, are admissible. If the level is not admissible but $k+h^\vee$ is still positive then one expects rich and interesting representation categories. However not much, if any, technology is known to study these representation categories. 
In this work, we restrict to special non-admissible levels for $\g =  \mathfrak{sl}_m$ and we consider the in many respects best behaved category, the Kazhdan-Lusztig category. The point is that a good interplay between conformal embeddings, minimal $W$-algebras and VOA extensions allows us to study this category. 

\subsection{The  category $\KL_k(\g)$}
The category $\mathcal O$ contains an important sub-category, called the Kazhdan-Lusztig category $\KL_k(\g)$ introduced in  \cite{KL},   where it was proved that $\KL_k(\g)$  is semi-simple for generic levels $k$. The Kazhdan-Lusztig category is the subcategory of the category $\mathcal O$ of ordinary modules, i.e. modules with finite-dimensional conformal weight spaces. 

Arakawa's result implies that $\KL_k(\g)$ is semi-simple for $k$ admissible. This in turn implies that  $\KL_k(\g)$  for $k$ admissible is a braided tensor category \cite{CHY} and in most cases rigidity is already proven \cite{C, CGL, CKoL}.
But $\KL_k(\g)$ could be also semi-simple for certain non-admissible levels. It turns out that minimal $W$-algebras help to answer the question of semi-simplicity of  $\KL_k(\g)$ at certain levels. A level is called collapsing, if the simple minimal $W$-algebra of $\g$, $W_k(\g, \theta)$, collapses to its affine vertex subalgebra.
It was proved in \cite{AKMPP-20} that $\KL_k(\g)$ is semi-simple when $k$ is a collapsing level or when the minimal affine $W$-algebra $W_k(\g, \theta)$ is rational. This result is used in \cite{CY} to construct vertex tensor categories at these levels.  In \cite{AMP}, it was proved that  $\KL_k(\g)$ could be also semi-simple when $\g$ is a Lie superalgebra and the level $k$ is collapsing.  

We note that there is a notion of collapsing levels for non-minimal affine $W$-algebras (cf. \cite{AEM, AMP22}), but the properties of the quantum Hamiltonian reduction functor are presently not well enough understood so that we can only conjecture that for non-minimal collapsing levels $\KL_k(\g)$ is semi-simple. 
This conjecture is true in one instance as proven in \cite{APV}.
The main result of  \cite{APV} shows that $\KL_k(\mathfrak{sl}_4)$ is semi-simple for $k=-5/2$. There it is also shown that $W_k(\mathfrak{sl}_4, f_{subreg})$ collapses to the rank one  Heisenberg vertex algebra $\mathcal H$. Therefore this  is the first example of non-minimal collapsing levels for which one  can prove that the category $\KL_k(\g)$ is semi-simple. Although the representation theory of $L_k(\mathfrak{sl}_4)$ is quite similar to those of minimal collapsing levels, the proof of semi-simplicity was  much more complicated. The proof uses explicit knowledge of very complicated formulas for singular vector in the universal affine vertex algebra $V^k(\mathfrak{sl}_4)$, calculation of $C_2$-algebra and information on  singular vectors in the generalized Verma modules at level $k$.  So it is very hard to use the same method to prove semi-simplicity of $\KL_k(\g)$ for other non-minimal collapsing levels. For us the knowledge of the structure of $\KL_k(\mathfrak{sl}_4)$  for $k=-5/2$ is very important as it is a major ingredient of the base case of the induction that proves our main Theorem. We also need results from conformal embeddings and vertex tensor categories. 

\subsection{Conformal embeddings and vertex  tensor categories } A main result of \cite{CY} is that if  $\KL(\g)$ is semisimple then it is a vertex tensor category in the sense of Huang-Lepowsky-Zhang \cite{HLZ0}--\cite{HLZ8}. Of course one not only wants to know the existence of the tensor structure but also more structure as fusion rules and rigidity. It turns out that explicit knowledge of conformal embeddings together with the orbifold theory of Robert McRae \cite{M} are sufficient for that in our cases. 

Another peculiarity of the vertex algebra $L_k(\mathfrak{sl}_4)$ is that it appears in the decomposition of conformal embedding
$\mathfrak{gl}_4 \hookrightarrow \mathfrak{sl}_5$ at level $k=-5/2$ from \cite{AKMPP-16}. Moreover, all irreducible modules from $\KL_k(\mathfrak{sl}_4)$ appears in the decomposition of this conformal embedding. Using this result, together with methods from \cite{CY, M}, one gets that $\KL_k(\mathfrak{sl}_4)$ is a rigid, braided tensor category. 

The situation for $\mathfrak{sl}_{m}$ at level $k = - \frac{m+1}{2}$  is very similar thanks to \cite{AKMPP-16}. In particular the decomposition of the conformal embedding 
 \bea\mathfrak{gl}_m \hookrightarrow \mathfrak{sl}_{m+1} \ \ \mbox{at level} \ k = - \frac{m+1}{2} \  \mbox{for}  \ m \ge 4 \label{conf-embed}
 \label{eq:ce}
  \eea
 is given there. 
  So one expects that $\KL _{k} (\mathfrak{sl}_m)$ for $k = - \frac{m+1}{2}$  is also  a rigid tensor category for all $m \ge 4$.  $L_{k} (\mathfrak{sl}_{m+1})$ for $k = - \frac{m+1}{2}$  is admissible if and only if $m$ is even, thus only for even $m$  one expects to get similar vertex tensor category structures as in the case $m=4$. The cases when $m$ is odd are more complicated and not all modules in $\KL_k(\mathfrak{sl}_m)$ appear in the decomposition of our conformal embedding.
Another motivation for studying levels $k = - \frac{m+1}{2}$ is that these levels are collapsing for certain non-minimal $W$-algebras (cf. \cite{AMP22}). In particular our main Theorem gives a whole family of examples of semisimple $\KL_k(\g)$ at collapsing levels $k$ for non-minimal $W$-algebras.

\subsection{Main results of the paper.}

Our main result is the structure of $\KL_k(\mathfrak{sl}_{m})$ for even $m \geq 4$ and $k=-\tfrac{m+1}{2}$.
\begin{main}\label{mainKL}
 $\KL_k(\mathfrak{sl}_{m})$ is a semi-simple, rigid braided tensor category for all even $m\ge 4$ and $k=-\tfrac{m+1}{2}$.
 Modules in $\KL_k(\mathfrak{sl}_{m})$ are simple currents, and all of them appear in the decomposition of conformal embedding (\ref{conf-embed}).
 \end{main}
  The proof of the theorem uses  minimal affine $W$-algebra $W_{k-1} (\mathfrak{sl}_{m+2}, \theta)$ and we prove the following properties. 
  \begin{main}\label{mainW} Assume that $m$ is even, $m \ge 4$.
 \begin{itemize} 
\item[(1)]  $W_{k-1} (\mathfrak{sl}_{m+2}, \theta)$ is a simple current extension of $L_{k}(\mathfrak{sl}_m) \otimes \mathcal H \otimes \mathcal M$, where $\mathcal H$ is the rank one Heisenberg vertex algebra, and $\mathcal M=\mathcal M(2)$ the singlet vertex algebra for $c=-2$

\item[(2)] Assume that $W$ is a highest weight  module from  $\KL_{k-1}(\mathfrak{sl}_{m+2})$. Then $H_{\theta} (W)$ is  an irreducible $W_{k-1} (\mathfrak{sl}_{m+2}, \theta)$--module  obtained  by induction from  $L_{k}(\mathfrak{sl}_m) \otimes \mathcal H \otimes \mathcal M$--module  $L_{k}(\mathfrak{sl}_m) \otimes \mathcal F^{\ell} _a \otimes \mathcal M_b$, where $ \mathcal F^{\ell} _a$ is an irreducible  $\mathcal H$--module, and $\mathcal M_b$ atypical $\mathcal M$--module (cf. Section \ref{main-sect} for notation).
\end{itemize}
\end{main}
Both these main Theorems are Theorem \ref{main} and the complete section \ref{main-sect} is devoted to its proof. The argument is an induction for $m$ and our base case is the main result of \cite{APV}. In order to do the induction step we need to recall some known result on Fock modules, the $p=2$ singlet algebra $\mathcal M$, the theory of simple current extensions of vertex algebras, conformal embeddings and minimal $W$-algebras. This is the content of sections \ref{secVOA} and \ref{Wetal}.
Using this background, the induction goes in several steps: 
\begin{enumerate} 
\item
Assume that Theorem \ref{mainKL} is true for some even integer $m$. Then one can use the simple currents together with the simple currents of the Heisenberg algebra $\mathcal H$ and those of the $p=2$ singlet algebra $\mathcal M$ to construct a certain simple current extension of $L_k(\mathfrak{sl}_m) \otimes \mathcal H \otimes \mathcal M$. This extension is a simple vertex algebra such that it has fields of conformal weight $3/2$ and in fact it is strongly generated by those together with the weight one fields of  $L_k(\mathfrak{sl}_m) \otimes \mathcal H$. There are strong uniqueness Theorems for vertex algebras of such type, e.g. a uniqueness Theorem of minimal $W$-algebras \cite[Theorem 3.1]{ACKL} and a uniqueness Theorem of hook-type $W$-algebras of type $A$ \cite[Lemma 6.3]{CL1}. The latter immediately (and the former with a bit of effort) tells us that this extension is 
 isomorphic to the simple minimal $W$-algebra $W_{k-1} (\mathfrak{sl}_{m+2}, \theta)$.
 \item We can now use the theory of vertex algebra extensions \cite{CKM1, CMY-limit} to study ordinary modules for this minimal $W$--algebra. In particular since simple modules of $\mathcal H$ are just Fock modules, since modules of $\mathcal M$ and their fusion rules are known \cite{CMY0, CMY} and since ordinary modules of $L_k(\mathfrak{sl}_m)$ are understood by induction hypothesis we can relatively easily determine all  ordinary modules for the minimal $W$-algebra, see Lemmata \ref{lower-bound} and \ref{lowerbound-tipical}. We find indecomposable modules. 
\item The next task is to determine those ordinary modules that come from minimal quantum Hamiltonian reduction of modules in $\KL_{k-1}(\mathfrak{sl}_{m+2})$. A necessary condition for this are certain weight and conformal weight conditions that we analyze next, see Lemma \ref{QHlowerbounded}.
\item  This result doesnot prove semisimplicity yet, but it is good enough to tell us that  $\KL_{k-1}(\mathfrak{sl}_{m+2})$ is of finite length and so we get that this is a vertex tensor category thanks to the main result of \cite{CY}, see Lemma \ref{KLcat}.
\item  Next, the existence of tensor category together with the conformal embedding \eqref{eq:ce} of \cite{AKMPP-16} and together with McRae's orbifold theory \cite{M} tell us that those simple modules of $\KL_{k-1}(\mathfrak{sl}_{m+2})$ that appear in the conformal embedding are all simple currents and we also get their fusion rules, see Corollary \ref{cor:sc}.
\item It remains to show that those simple modules that appear in the conformal embedding are all simple modules in $\KL_{k-1}(\mathfrak{sl}_{m+2})$. 
Recall that  a necessary condition for modules coming  from minimal quantum Hamiltonian reduction  are certain weight and conformal weight conditions and there are a few possibilities that we now need to exclude. We prove that if these additional possibilities were coming from  $\KL_{k-1}(\mathfrak{sl}_{m+2})$ via minimal quantum Hamiltonian reduction then this would contradict certain fusion rules, i.e. we get a contradiction, see Lemma \ref{nije+-1}. 
\item All these results together conclude our proof of the induction step and that is summarized in Theorem \ref{induction step}. 
\end{enumerate}
 
 We also present a  complete list of irreducible ordinary modules for $W_{k-1} (\mathfrak{sl}_{m+2}, \theta)$ in Section \ref{Word}. We prove:  
\begin{itemize}
\item Each irreducible, ordinary   $W_{k-1} (\mathfrak{sl}_{m+2}, \theta)$--module is obtained  by induction from  $L_{k}(\mathfrak{sl}_m) \otimes \mathcal H \otimes \mathcal M$--module  $U^{n} _i \otimes \mathcal F^{\ell} _a \otimes \mathcal L_{\nu}$, where $ U^{n} _i $ belongs to  $KL_k(\mathfrak{sl}_m)$, $\mathcal F^{\ell} _a$ is an irreducible  $\mathcal H$--module, and $\mathcal L_{\nu} $ is irreducible  $\mathcal M$--module (typical or atypical).
\item Moreover, the category of ordinary  $W_{k-1} (\mathfrak{sl}_{m+2}, \theta)$--modules is not semi-simple.
\end{itemize}

 We conclude this work with an alternative proof of the decomposition of the minimal $W$--algebra into modules of the affine subVOA times $\mathcal M$ in section \ref{Wdecomp}.  Moreover, we prove in Theorem   \ref{w-decomp-nova}
  that our minimal affine $W$-algebra admits the following realization:
  \begin{itemize}
  \item $W_{k-1} (\mathfrak{sl}_{m+2}, \theta) = \mbox{Com}( L_{k} (\mathfrak{sl}_{m+1})  \otimes \mathcal S)$, where $\mathcal S$ is the $\beta \gamma$ vertex algebra.
  \end{itemize}
 The advantage of this second argument here is that it also holds for odd $m$.

\subsection{Outlook}

As outlined before, our present work was made possible by  various quite recent results of our research teams. These are a good understanding of conformal embeddings \cite{AKMPP-16, AKMPP-JA, AKMPP-17,AKMPP-18, AMPP-20}, collapsing levels \cite{AKMPP-20, AMP22, AEM}, its implications on vertex tensor categories \cite{CY} and finally a full understanding of the representation theory of the singlet algebras \cite{CMY0, CMY}. 
 
Previously some semisimplicity and tensor category results for $\KL_k(\g)$ at rational but non-admissible levels had been obtained, but they always used collapsing levels directly or that the minimal $W$-algebra is rational. The novelty of our work in this direction is that the minimal $W$-algebra has a quite complicated representation theory, but nonetheless it was still possible to use it to show semisimplicity of $\KL_k(\g)$. 

Already the singlet algebras times a Heisenberg algebra allows for many interesting VOA extensions \cite{CRW, CR2} and we expect that more affine VOAs at rational and non-admissible levels times suitable singlet algebras and times a Heisenberg VOA extend to certain $W$-algebras. This deserves further study, but as mentioned before,  the general quantum Hamiltonian reduction isn't as well understood as the minimal reduction and so it will be considerably more difficult to infer interesting statements about $\KL_k(\g)$ in those cases.


\subsection*{Acknowledgements}

 The work of T.C. is supported
by NSERC Grant Number RES0048511.  D.A., O. P. and I. V. are  partially supported by the
QuantiXLie Centre of Excellence, a project cofinanced
by the Croatian Government and European Union
through the European Regional Development Fund - the
Competitiveness and Cohesion Operational Programme
(KK.01.1.1.01.0004).

  \section{Underlying vertex algebras and their categories}\label{secVOA}

  \subsection{Heisenberg vertex algebra $\mathcal H$}
  \label{Heisenbeg-voa}
  Let $\mathcal H$ be the rank one Heisenberg vertex algebra generated by  $h = h(-1){\bf 1}$,  and standard conformal vector $\omega_{st} = \frac{1}{2} h(-1) ^2 {\bf 1}$. For $\alpha \in {\C}$, let $\mathcal F_{\alpha}$ be the Fock $\mathcal H$--modules generated by highest weight vector $v_{\alpha}$ such that
  $$ h(n) v_{\alpha} = \delta_{n,0} \alpha v_{\alpha} \quad (n \ge 0). $$
  The lowest conformal weight  of $\mathcal F_{\alpha}$ with respect to $\omega_{st}$ is $\frac{1}{2} \alpha ^2$.
 It is sometimes convenient to rescale $h \mapsto \sqrt{\ell} h$. We then denote by $\mathcal F^\ell_\alpha$ the Fock module of weight $\alpha$ for the rescaled Heisenberg field. 
 Its lowest conformal weight is then $\frac{\alpha^2}{2 \ell}$.  
  Let $\mathcal C^\ell_{\mathcal F}$ be the category of Fock modules with real weight with respect to the rescaled Heisenberg field. It is a vertex tensor category \cite[Theorem 2.3]{CKLR} with fusion rules
  \[
  \mathcal F_{\alpha}^\ell \times  \mathcal F_{\beta}^\ell =  \mathcal F_{\alpha+ \beta}^\ell
  \]
  and braiding \cite[Remark 2.13]{CGR} or \cite[Proposition 3.11]{ALSW} 
  \[
  c_{ \mathcal F_{\alpha}^\ell,  \mathcal F_{\beta}^\ell} = e^{\pi i \frac{\alpha\beta}{\ell}} \text{Id}_{ \mathcal F_{\alpha+\beta}^\ell}.
  \]

 \subsection{Singlet vertex algebra $\mathcal M$}
 
  One of our main building blocks in our paper is the singlet algebra $\mathcal M(2)$, whose category $\mathcal O_{\mathcal M}$ of finite-length grading-restricted generalized modules is completely understood thanks to \cite{CMY0,CMY}.
The singlet vertex algebra $\mathcal M(p)$   is  a non-rational vertex algebra whose representation theory has  attracted a lot of interest  (cf. \cite{A-2003, AdM-2007, CM14, CMW}). In the case $p=2$, the singlet algebra is isomorphic to the $W$-algebra $W (\mathfrak{sl}_3, f_{princ})$ at central charge $c=-2$ (cf. \cite{Wa}).
 
  We will recall the necessary results and introduce a convenient notation. Firstly we set $\mathcal M = \mathcal M(2)$ for the singlet VOA itself. 
 Two classes of modules of $\mathcal M$ are denoted by  $\mathcal M_{r, 1}$ for $r \in \mathbb Z$ and $\mathcal F_{\nu}$ for $\nu \in \mathbb C$.
 These modules are all simple, except for the $\mathcal F_{\nu}$ and $\nu$ integer, and they give a complete list of inequivalent simple modules. 
 We want to avoid confusion with Fock modules and so we will use the notation $\mathcal M_{r-1} := \mathcal M_{r, 1}$ and $\mathcal V_{\nu}:= \mathcal F_{-\nu}$. 
 With this notation the fusion rules take the convenient form  (cf. \cite{AM17, CMY}):
 \begin{equation}
 \begin{split}
 \mathcal M_i \times \mathcal M_j &= \mathcal M_{i+j}, \\
  \mathcal M_i \times \mathcal V_\nu &= \mathcal V_{i+\nu} .
 \end{split}
 \end{equation}
 If $\nu = i \in \mathbb Z$, then there is a non-split exact sequence
 \begin{equation}\label{eq:ses}
 0 \rightarrow \mathcal M_i \rightarrow \mathcal V_i \rightarrow \mathcal M_{i+1} \rightarrow 0.
 \end{equation}

  The conformal weights $\Delta$ of the top levels of these modules are
  \[
  \Delta(\mathcal M_i) = \frac{|i|(|i|+1)}{2}, \qquad 
  \Delta(\mathcal V_\nu) = \frac{  \nu   (  \nu  +1)}{2}. 
  \]
  
  Note that the formulas for conformal weights are different, and only for atypical modules it uses the absolute value.   The reason is in the fact that for $i \ge 0$, $\mathcal M_i$ is generated by singular vector of the Fock module, but for $i <0$, $\mathcal M_i$ is generated by different Virasoro singular vector inside of  the Fock module (cf. \cite{AdM-2007}).
  
  The singlet algebra $\mathcal M$ allows for interesting extensions, most importantly the simple current extensions 
  \[
 \mathcal A=  \bigoplus_{i \in \mathbb Z} \mathcal M_i \qquad \text{and} \qquad  \mathcal W=  \bigoplus_{i \in 2\mathbb Z} \mathcal M_i \
  \]
  are the well-known symplectic fermion algebra $\mathcal A$ and its even subalgebra, the $p=2$ triplet algebra $\mathcal W$ (cf. \cite{Abe, AM08, Ka91}). 
  Extensions of a vertex superalgebra in a vertex tensor category are in one-to-one correspondence with commutative superalgebra objects in the category \cite{CKL}. In particular the braiding on the simple currents satisfies $c_{\mathcal M_i, \mathcal M_j} = (-1)^{ij} \text{Id}_{\mathcal M_{i+j}}$. 
  

  \subsection{The auxillary VOA $\mathcal U$}\label{auxVOA}
  
  Let $\mathcal U^n$ be a VOA with a semisimple vertex tensor category of modules $\mathcal C_{\mathcal U}^n$ whose simple objects are $\mathcal U^n_i$ for integer $i$, fusion rules
  \[
  \mathcal U^n_i \times \mathcal U^n_j = \mathcal U^n_{i+j}
  \]
  and braiding
  \[
  c_{\mathcal U^n_i, \mathcal U^n_j} = e^{ \pi i  \frac{2ij}{m}} \text{Id}_{\mathcal U^n_{i+j}}
  \]
  where is $m= 2n$. We assume that the conformal weight of the top level of $U^n_i$ is $\frac{i^2}{m} + |i|$ and we also assume all these modules to be ordinary.
  
  \subsection{Simple current extensions}\label{extVOA}
  
  We consider the extension 
  \[
  \mathcal W = \bigoplus_{i \in \mathbb Z} \mathcal U_i^n \otimes \mathcal F_i^\ell \otimes \mathcal M_i
  \]
  for $\ell =  -\frac{m}{m+2}$, $m =2n$ and $n$ a positive integer at least equal to two. The conformal weight of the top level of the 
  summand $\mathcal W_i = \mathcal U_i^n \otimes \mathcal F_i^\ell \otimes \mathcal M_i$ is
    \[
   \Delta(\mathcal W_i) = \frac{i^2}{m} + |i| - i^2 \frac{m+2}{2m} + \frac{i^2 + |i|}{2} = \frac{3}{2} |i|.
  \]
  $\mathcal W$ is thus a (super)VOA extension \cite{CKL} and we need to determine if it is even or not.  
 Let $\mathcal C$ be the Deligne product $\mathcal C_{\mathcal U}^n \boxtimes \mathcal C^\ell_{\mathcal F} \boxtimes \mathcal O_{\mathcal M}$. It is a vertex tensor category of $\mathcal U^n \otimes \mathcal H \otimes \mathcal M$--modules \cite[Theorem 5.5]{CKM2} and the braiding satisfies \cite[Remark 5.3]{CKM2}
\[
c_{\mathcal U_i^n \otimes \mathcal F_i^\ell \otimes \mathcal M_i, \mathcal U_j^n \otimes \mathcal F_j^\ell \otimes \mathcal M_j} = 
c_{\mathcal U_i^n, \mathcal U_j^n} \otimes c_{\mathcal F_i^\ell, \mathcal F_j^\ell}  \otimes  c_{\mathcal M_i, \mathcal M_j} = \text{Id}_{\mathcal U_{i+j}^n \otimes \mathcal F_{i+j}^\ell \otimes \mathcal M_{i+j}}.
\]  
 All braidings are trivial, i.e. $\mathcal W$ is even.
  
Vertex algebra extensions in a vertex tensor category are in one-to-one correspondence with commutative algebra objects in the category \cite{HKL}. 
The category of modules of the extended VOA then coincides as a braided tensor category with the category of local modules of this commutative algebra \cite{CKM1}.
In our case $\mathcal W$ is only an object of the direct limit completion of $\mathcal C$, which we denote by $\overline{\mathcal C}$. This is still a vertex tensor category, extensions can still be identified with commutative algebras as well as the category of modules of the extended VOA is identified with the category of local modules of the commutative algebra \cite{CMY-limit}. 
The case of simple current extensions has been studied exhaustively and our problem is quite similar to the ones of  \cite{ACKR, CR, CMY2}, which we will now recall. 
Introduce the short-hand notation $J_i := \mathcal U_i^n \otimes \mathcal F_i^\ell \otimes \mathcal M_i$ and $\mathcal C_W$ for the category of modules of the VOA $\mathcal W$. 
Let $\overline{\mathcal D}$ be the subcategory of $\overline{\mathcal C}$ whose indecomposable objects $M$ satisfy 
\[
h_{J_1 \times M} - h_{M}  \in \mathbb Z  
\]
where $h_{J_1 \times M}$ and $h_{M}$ denote the conformal weights of the top levels of $J_1 \times M$ and $M$. 
There is then an induction functor $\Ind$ from $\overline{\mathcal C}$ to not necessarily local modules of $\mathcal W$. 
For an object $M$  in $\overline{\mathcal C}$ we have
\[
\Ind(M) \cong_{\overline{\mathcal C}} \mathcal W \times M = \bigoplus_{i \in \mathbb Z} J_i \times M, 
\]
i.e. $\Ind(M)$ is  $\mathcal W \times M$  together with a certain action of $\mathcal W$ that is not relevant for this work. 

A necessary condition for the induction of a module $M$ to be local, i.e. to be in $\mathcal C_W$, is that $M$ is in  $\overline{\mathcal D}$. 
If $J_i \times M \cong M$ for a simple module $M$ if and only if $i=0$, which is clearly true in our case, then 
for simple modules this is also a sufficient condition 
and one gets all simple modules in this way.
\begin{proposition}\cite[Prop. 3.2]{CMY2}
An object $N$ of $\mathcal C_W$ is simple if and only if there exists a simple module $M$ in $\overline{\mathcal D}$ with $\Ind(M) \cong N$. 

Morover, for simple modules $M, M'$ in $\overline{\mathcal D}$, $\Ind(M) \cong \Ind(M')$ if and only if there exists $i \in \mathbb Z$ with $M' \cong J_i \times M$. 
\end{proposition}
 The condition $M$ in    $\overline{\mathcal D}$ is equivalent to 
 \[
  \bigoplus_{i \in \mathbb Z} J_i \times M
 \]
  being half-integer graded with $h_{J_i \times M} = h_{M} \mod \mathbb Z$ if $i$ is even and 
$h_{J_i \times  M} = h_{M} + \frac{1}{2} \mod \mathbb Z$ if $i$ is odd.

 \section{Minimal $\mathcal W$-algebras and conformal embeddings}\label{Wetal}

\subsection{The minimal $\mathcal W$-algebra of $\mathfrak{sl}_{m+2}$} 
Let $\g = \mathfrak{sl}_{m+2}$, $-\theta$ the minimal root, $W_k(\g, \theta)$ the associated  minimal affine $W$-algebra (cf. \cite{KW}, \cite{Araduke},  \cite{AKMPP-JA}).
 
The minimal $\mathcal W$-algebra has an affine vertex algebra  $\widetilde{L}_{k+1}(\mathfrak{gl}_m )  =\widetilde{L}_{k+1}(\mathfrak{sl}_m ) \otimes \mathcal H$ as subalgebra, where $\widetilde{L}_{k+1}(\mathfrak{sl}_m )$ is a certain quotient of the universal affine vertex algebra  $V^{k+1}(\mathfrak{sl}_m )$.
(Note that in general we don't know if  $\widetilde{L}_{k+1}(\mathfrak{sl}_m )$ is simple).
 
Let $\omega_1, \dots, \omega_{m+1}$ be the fundamental weights for $\g$. Let $\bar \omega_1, \dots, \bar  \omega_{m-1}$ be the fundamental weights  for $\mathfrak{sl}_m$.  One shows that
 $\langle \omega_ i,  \omega_j  \rangle = \text{min}\{ i, j\} - \frac{ij}{m+2}$.

For the dominant integral weight
$\lambda = \lambda_1 \omega_1 + \cdots+ \lambda_{m+1} \omega_{m+1}$, $\lambda_i \in {\Z}_{\ge 0}$, for $\g$, we define the dominant integral weight for $\mathfrak{sl_m}$ with
$\bar \lambda = \lambda_2 \bar \omega_1 + \cdots + \lambda_{m} \bar \omega_{m-1}$.

The Heisenberg field $J$  is obtained from 
$ J \equiv  \frac{1}{m+2} t(-1){\bf 1} $, where  $$ t =  \mbox{diag} ( m,  \underbrace{-2, \dots, -2}_{m}, m ) \in \g. $$Then
$J(1) J =  -\frac{m}{m+2} {\bf 1}$(cf. \cite[(3.20)]{CL1}).
Note that $t$ acts on highest weight vector $V_{\g}(\lambda)$ as $ (m+2)  \langle \lambda,  \omega_1 - \omega_{m+1} \rangle  \mbox{Id}$.


 There is a  Quantum Hamiltonian reduction  functor $H_{\theta}$ from $\KL_k(\g)$ 
to a certain  subcategory  of $W_{k}(\g, \theta)$--modules 
 whose properties
will be very important in the following. Details on the properties of this functor can be found in \cite{Araduke} and \cite{KW}, see also \cite{AKMPP-20}. We have
\begin{theorem}\label{thm:min}\cite{Araduke}
$H_\theta$ is an exact functor from the category $\KL_k(\g)$ to the category of ordinary modules of $W_{k}(\g, \theta)$.
\end{theorem}
and 
\begin{proposition}\label{prop:min} \cite{Araduke, KW} Assume that  $k \notin {\Z}_{\ge 0}$,  $\lambda$ is dominant integral weight for  $\mathfrak{sl}_{m+2}$ such that $\tilde L_k(\lambda)$  a  highest weight module in $\KL_k(\mathfrak{sl}_{m+2})$. If $\tilde L_k(\lambda)$  is simple then so is $H_{\theta} (\tilde L_k(\lambda))$.
Moreover  $H_{\theta} (\tilde L_k(\lambda))$ is a highest weight module for $W_k(\mathfrak{sl}_{m+2}, \theta)$ with a highest weight vector $\overline {v}_{\lambda}$ such that the highest weight with respect to $\widetilde{L}_{k+1}(\mathfrak{sl}_m ) \otimes \mathcal H$ (i.e., the $\mathfrak{gl}_m$ highest weight) is $(  \langle \lambda, \omega_1 - \omega_{m+1} \rangle , \bar \lambda)$. In particular we have:
 \begin{itemize}
 \item $J(0) \overline {v}_{\lambda} =     \langle \lambda, \omega_1 - \omega_{m+1} \rangle   \overline {v}_{\lambda}$. 
 \item  $H_{\theta} (\tilde L_k(\lambda))_{top}$ contains a  $\mathfrak{sl}_{m}$--submodule which is isomorphic to the irreducible, highest weight module  $V_{\mathfrak{sl}_m} (\bar \lambda) = \mathfrak{sl}_{m} . \overline {v}_{\lambda}$  with highest weight $\bar \lambda$.
 \item  $L(0)   \overline {v}_{\lambda} = \left( \frac{ \langle \lambda, \lambda + 2 \rho\rangle}{ 2( k + h^{\vee})} -  \tfrac{1}{2}  \langle  \lambda, \theta \rangle \right)    \overline {v}_{\lambda}. $  
   \end{itemize}

\end{proposition}


  Note that if the module is   the quantum Hamiltonian reduction of an object in $\KL_k(\g)$, its Heisenberg weight   needs to lie in $\frac{1}{m+2} \mathbb Z$.

\subsection{The category $\KL_k(\mathfrak{sl}_4)$}\label{KL4}
Let $m \in {\Z}_{>0}$, $m \ge 4$.
We introduce the following notation for the irreducible $L_{-\frac{m+ 1}{2} }(\mathfrak{sl}_m)$--modules:
$$U_i ^{(m/2)} =L_{-\frac{m+1}{2}}(  i \omega_1), \ U^{(m/2)} _{-i}=L_{-\frac{m+1}{2}}(  i \omega_{m-1}), \ i \in \mathbb{Z}_{\geq 0}. \vspace*{2mm}$$
 We know from \cite{APV} that:
 
\begin{itemize}
 \item  $\KL_k(\mathfrak{sl}_4)$ is a semisimple rigid braided tensor category for $k=-5/2$.
 \item The set  $\{ U_i ^{(2)} , i \in {\Z}\}$ gives all irreducible modules in   $\KL_k(sl(4))$.
 \item  The  modules $U_i ^{(2)} $ are simple current with the following fusion rules $$ U_i ^{(2)} \times U_j ^{(2)} = U_{i+j}^{(2)}, \quad i, j \in {\Z}. $$

\end{itemize}

\subsection{Conformal embedding of $L_{-  \frac{m+1}{2}}(\mathfrak{gl}_{m})$ in  $L_{ - \frac{m+1}{2}}(\mathfrak{sl}_{m+1})$  }
In our paper we shall need the following result on conformal embedding:

\begin{theorem} \label{dec-akmpp}  \cite{AKMPP-16}.The vertex algebra  $  L_{- \frac{m+1}{2}}(\mathfrak{gl}_m)=   L_{- \frac{m+1}{2}}(\mathfrak{sl}_m) \otimes \mathcal H$ is conformally embedded into the  simple affine vertex algebra $L_{-  \frac{m+1}{2}}(\mathfrak{sl}_{m+1})$ and we have the following decomposition:

\[
L_{- \frac{m+1}{2}}(\mathfrak{sl}_{m+1}) = \bigoplus_{i \in \mathbb Z} U_i^{(m/2)} \otimes \mathcal F_i^s
 \]
 with $s = - \frac{m}{2}$. In particular, $U_i^{(n)}$, $i \in {\Z}$, are simple $U_0^{(n)}$--modules.
 \end{theorem}
 
 Note that lowest conformal weight of  $\mathcal F_i^s$ is $-\frac{i^2}{m}$, so the conformal weight  of top component of  $U_i^{(n)} \otimes \mathcal F_i^s$ is $\vert i \vert $.

A special case of the main result of \cite{M} is 
\begin{theorem}\label{thm:robert}
Assume that $\KL_{- \frac{m+1}{2}}(\mathfrak{sl}_m)$ is a vertex tensor category. Then the $U_i^{(n)}$, $i \in {\Z}$, are simple currents, in particular rigid, and they form a tensor category that is braid reversed equivalent to the category of the Fock modules  $\mathcal F_i^s, i \in \Z$. 
\end{theorem} 
 
 





  \section{Minimal $W$-algebra $W_{k-1}(\mathfrak{sl}_{2n+2}, \theta) $  and tensor category $\KL_k(\mathfrak{sl}_{2n})$ at the level  $k =-n-\frac{1}{2}$}
  \label{main-sect}
  In this section we have the following setting:
  \begin{itemize}
  \item  $n$ is a positive integer starting at $2$,  $m= 2n$;
  \item $k_{n} = -n - \frac{1}{2}= - \frac{m+1}{2}$.
  \item $\g_n = \mathfrak{sl}_{2n}$.
  \item $\theta^{(n)}$ is the highest root in $\g_n$.
  \end{itemize}

     \begin{theorem} \label{main} For each $n \in \mathbb Z_{\geq 2}$
 \begin{itemize}
 \item[(1)]   the set  $\{ U_{i} ^{(n)}  \vert \ i \in {\Z} \}$ provides a complete list of inequivalent  irreducible modules in $\KL_{k_n} (\g_n)$ with fusion rules
 $$ U_i ^{(n)} \times U_j ^{(n)} = U_{i+j}^{(n)}, \quad i, j \in {\Z}.$$
 \item[(2)]  $\KL_{k_n}(\g_n)$ is a semi simple rigid braided tensor category.
  \item[(3)] There exists conformal embedding $L_{k_n} (\g_n) \otimes \mathcal H \otimes \mathcal M(2) \hookrightarrow  W_{k_{n+1}}(\g_{n+1}, \theta^{(n+1)})$.
 \item[(4)] $W_{k_{n+1}}(\g_{n+1}, \theta^{(n+1)})$ is a simple current extension of $ L_{k_n} (\g_n) \otimes \mathcal H \otimes \mathcal M(2)$ and we have the following decomposition:
 $$W_{k_{n+1}}(\g_{n+1}, \theta^{(n+1)}) = \bigoplus_{i \in {\Z}}  U_{i} ^{(n)} \otimes \mathcal F^\ell_{i} \otimes \mathcal M_i$$
 where $\ell = - \frac{m}{m+2}$.
 \end{itemize}
 \end{theorem}
 The Theorem  will be proven by induction for $n$. The case $n=2$, i.e. $m=4$ is known and recalled in section \ref{KL4}.
 
 We note that the induction hypothesis point (1) and (2) says that $ L_{k} (\g_n)$ is a realization of  the auxillary vertex algebra $\mathcal U^n$ of section \ref{auxVOA}.
We thus have the commutative algebra object discussed in section \ref{extVOA}
 \[
 A = \bigoplus_{i \in {\Z}}  U_{i} ^{(n)} \otimes \mathcal F^\ell_{i} \otimes \mathcal M_i.
 \]
 The conformal weight of the top level of $A_i = U_{i} ^{(n)} \otimes \mathcal F^\ell_{i} \otimes \mathcal M_i$ is
 \[
 \Delta(A_i) = \frac{i^2}{m} + |i| - i^2 \frac{m+2}{2m} + \frac{i^2 + |i|}{2} = \frac{3}{2} |i|.
 \]

 \begin{lemma}\label{hook identify}
 $A \cong W_{k_{n+1}}(\g_{n+1}, \theta^{(n+1)})$.
 \end{lemma}
 \begin{proof}
We want to determine the generating type of $A$. As a vertex algebra $A$ is generated by $A_0,  A_1$ and $A_{-1}$ (this follows for example from Main Theorem 1 of \cite{CKM2}). 
$A_0$ is strongly generated by the weight one subspace of  $L_{k_n} (\g_n)$, the Heisenberg field $h$, the Virasoro field, call it $L$ of $\mathcal M$ and the singlet field, call it $W$, of $\mathcal M$ of conformal weight three. 

Let us describe these modules in terms of symplectic fermions, these are  $\bigoplus_{i \in \mathbb Z} \mathcal M_i$ as an extension of $\mathcal M$.
Symplectic fermions are strongly and freely generated by two odd fields $\chi^\pm$ of conformal weight one with OPE
\[
\chi^+(z) \chi^-(w) = \frac{1}{(z-w)^{2}}.
\]
The singlet field $W$ is just $\chi^+ \partial \chi^-$.
Symplectic fermions are $\mathbb Z$-graded by giving $\chi^\pm$ degree $\pm 1$. 
The subspace of degree $\pm 1$ is then precisely the singlet module $\mathcal M_{\pm 1}$. Denote by $| \, 0 \, \rangle$ the vacuum of $\mathcal M$. Then the top space of $\mathcal M_{\pm 1}$ corresponds to the vector $\chi^\pm_{-1}| \, 0 \, \rangle$. Similar the subspaces of conformal weight two and three of $\mathcal M_{\pm 1}$ are clearly one-dimensional spanned by $\chi^\pm_{-2}| \, 0 \, \rangle$, respectively $\chi^\pm_{-3}| \, 0 \, \rangle$. In particular these states correspond to derivatives of $\chi^\pm(z)$. This means that any field in $\mathcal M_{\pm 1}$ is a normally ordered polynomial of $L, W$ and their iterated derivatives times $\chi^\pm$.
From this consideration we learn that the set $\mathcal S$, consisting of the strong generators of $\mathcal W_0$ together with the fields corresponding to the top level of $\mathcal W_{\pm 1}$ already generate $\mathcal W$ as a VOA. Since the conformal weight of the top level of  $\mathcal W_{\pm 1}$ is $\frac{3}{2}$ and since the one of  $\mathcal W_{\pm 2 }$ is $3$ these must close under OPE.
  It follows immediately from uniqueness of hook-type $W$-algebras of type $A$ \cite[Theorem 9.1]{CL1} that $A \cong W_{k_{n+1}}(\g_{n+1}, \theta^{(n+1)})$.
  \end{proof}

 Thus claims (3) and (4) follow from claims (1) and (2). 
 
 Let us prove claims (1) and (2) by induction.
 The induction claim is now that $\KL_{k_n}(\g_n)$ is a semi simple rigid braided tensor category and a complete list on inequivalent simple objects of $L_{k_n}(\g_n)$ are the $U_i^{(n)}$ and they are simple currents and the categorical dimension of any of these are one.

   We are interested in the category of local $A$--modules that are ordinary. There are two series, which we call atypical and typical. 
 We start with the atypical case, these are obtained via induction from $U_{a} ^{(n)} \otimes \mathcal F^\ell_{b} \otimes \mathcal M_c$ and since
 \[
 Ind  (U_{a} ^{(n)} \otimes \mathcal F^\ell_{b} \otimes \mathcal M_c) \cong  Ind (U_{a+1} ^{(n)} \otimes \mathcal F^\ell_{b+1} \otimes \mathcal M_{c+1}) ,
 \]
 it is enough to consider
 \[
 \mathcal A_{a, b} :=
Ind (U_{0} ^{(n)} \otimes \mathcal F^\ell_{a} \otimes \mathcal M_b) = \bigoplus_{i \in {\Z}}  U_{i} ^{(n)} \otimes \mathcal F^\ell_{i+a} \otimes \mathcal M_{i+b}.
 \]
 This is a local module if and only if it is half-integer (and not integer) graded by conformal weight. This is easily checked to happen if and only if $a \frac{m+2}{m} \in \mathbb Z$. 
 The top component  of $U_{i} ^{(n)} \otimes \mathcal F^\ell_{i+a} \otimes \mathcal M_{i+b}$ has conformal weight
 \begin{equation}
 \begin{split}
 \Delta(a, b, i) &= \frac{i^2}{m} + |i| - (a+i)^2 \frac{m+2}{2m} + \frac{(b+i)^2}{2} + \frac{|b+i|}{2}  \\
 &= |i| +  \frac{|b+i|}{2}  +bi - ai \frac{m+2}{m}  +\frac{b^2}{2} - a^2\frac{m+2}{2m} \\
 &= |i| +  \frac{|b+i|}{2}  + (b- a \frac{m+2}{m} ) i   +\frac{b^2}{2} - a^2\frac{m+2}{2m}
 \end{split}
 \end{equation}
 From this we see  the following lemma:
 
 \begin{lemma} \label{lower-bound} 
 	Assume that the induced module $\mathcal A_{a,b}$ is local. Then $\mathcal A_{a,b}$  is lowerbounded if and only if $| b - a\frac{m+2}{m}| \leq \frac{3}{2}$. 
   We have the following cases:

 \begin{itemize}
 \item  $b - a\frac{m+2}{m} =0$. Then  $\Delta(a, b, i) =  |i| +  \frac{|b+i|}{2} +\frac{b^2}{2} - a^2\frac{m+2}{2m}$. The  lowerbound is attained for $i=0$.
 \item  $b - a\frac{m+2}{m} =1$. Then  $\Delta(a, b, i) =  |i| + i + \frac{|b+i|}{2} +\frac{b^2}{2} - a^2\frac{m+2}{2m}$. The  lowerbound is attained for $i + b=0$ if $b \ge  0$ and for $i =0$ if $b < 0$.
  \item  $b - a\frac{m+2}{m} =-1$. Then  $\Delta(a, b, i) =  |i| - i + \frac{|b+i|}{2} +\frac{b^2}{2} - a^2\frac{m+2}{2m}$. The  lowerbound is attained for $i + b=0$ if $b <  0$ and for $i =0$ if $b \ge  0$.
 
  \end{itemize}
  \end{lemma}
 So we conclude that the  lowerbound is attained either for $i=0$ or for $i+b=0$.

  We turn to the typical case, these are obtained via induction from $U_{a} ^{(n)} \otimes \mathcal F^\ell_{\mu} \otimes \mathcal V_\nu$ and since
 \[
 Ind  (U_{a} ^{(n)} \otimes \mathcal F^\ell_{\mu} \otimes \mathcal V_\nu ) \cong  Ind (U_{a+1} ^{(n)} \otimes \mathcal F^\ell_{\mu+1} \otimes \mathcal V_{\nu+1}) ,
 \]
 it is enough to consider
 \[
 \mathcal T_{\mu, \nu} :=
Ind (U_{0} ^{(n)} \otimes \mathcal F^\ell_{\mu} \otimes \mathcal V_\nu) = \bigoplus_{i \in {\Z}}  U_{i} ^{(n)} \otimes \mathcal F^\ell_{i+\mu} \otimes \mathcal V_{i+\nu}.
 \]
 This is a local module if and only if it is half-integer (and not integer) graded by conformal weight. This is easily checked to happen if and only if $\nu = \frac{m+2}{m} \mu  \mod \mathbb Z$. 
 The top level of $U_{i} ^{(n)} \otimes \mathcal F^\ell_{\mu+i} \otimes \mathcal V_{\nu+i}$ has conformal weight

   \begin{equation} \label{tip-low}
 \begin{split}
 \Delta(\mu, \nu, i) &=  -\frac{\mu^2(m+2)}{2m} +\frac{\nu(\nu+1)}{2} + |i| + i \left( \nu - \frac{m+2}{m} \mu   +\frac{1}{2} \right).
 \end{split}
  \end{equation}
 From this we see:
 
 \begin{lemma}   \label{lowerbound-tipical} 
  Assume that the induced module $\mathcal T_{\mu, \nu}$ is local. Then $\mathcal T_{\mu, \nu}$  is lowerbounded if and only if  $\nu = \frac{m+2}{m} \mu$ or $\nu = \frac{m+2}{m} \mu -1 $. The  lowerbound is attained for $i=0$.
   
\end{lemma}
\begin{proof}
From (\ref{tip-low}) we see that this module is lowerbounded if and only if  $\vert \nu - \frac{m+2}{m} \mu   +\frac{1}{2} \vert \le 1$. Since  $\nu = \frac{m+2}{m} \mu  \mod \mathbb Z$, we have only two possibilities
  \begin{itemize}
 \item  $\nu  - \mu \frac{m+2}{m} =0$. Then  $\Delta(\mu, \nu , i) =  |i| + \frac{i}{2}   -\frac{\mu^2(m+2)}{2m}  +\frac{\nu(\nu+1)}{2}$. The  lowerbound is attained for $i=0$.
 \item  $\nu  - \mu \frac{m+2}{m} =-1$. Then  $\Delta(\mu , \nu , i) =  |i|  - \frac{i}{2}     -\frac{\mu^2(m+2)}{2m}  +\frac{\nu(\nu+1)}{2}$.
  The  lowerbound is attained  for $i =0$.
  \end{itemize}
The proof follows.
\end{proof} 
 
 Next, we wonder if $\mathcal A_{a, b}$ and $ \mathcal T_{\mu, \nu}$ are  the quantum Hamiltonian reductions of  ordinary modules of $L_{k_{n+1}}(\g_{n+1})$.

 \begin{lemma}\label{QHlowerbounded} Let $L_{k_{n+1}}(\lambda)$ be an  ordinary module of $L_{k_{n+1}}(\g_{n+1})$.
 If  $\mathcal A_{a, b}$ is the quantum Hamiltonian reduction of $L_{k_{n+1}}(\lambda)$,
 then
 the lower bound of an induced module is attained for $i=0$. In particular, we have that $\lambda_2 =0 = \lambda_{m}$.
 
 If  $\mathcal T_{\mu, \nu}$ is the quantum Hamiltonian reduction of an ordinary module of $L_{k_{n+1}}(\g_{n+1})$, then $\nu \in \mathbb Z$.
 \end{lemma}
 Since the $\mathcal T_{\mu, \nu}$ for $\nu \in \mathbb Z$ are not simple but satisfy \eqref{eq:ses}, we get using Theorem \ref{thm:min} and Proposition \ref{prop:min} the following corollary:
 \begin{corollary}${}$ \label{cor:Aab}
 \begin{itemize}
 \item[(1)]
 Let $H_{\theta^{(n+1)}} (L_{k_{n+1}}(\lambda))$ be the quantum Hamiltonian reduction of a simple module in $\KL_{k_{n+1}}(\g_{n+1})$. Then there exists $a, b$  as above with $H_{\theta^{(n+1)}} (L_{k_{n+1}}(\lambda)) \cong \mathcal A_{a, b}$.
 \item[(2)] Let $H_{\theta^{(n+1)}} (\tilde L_{k_{n+1}}(\lambda))$ be the quantum Hamiltonian reduction of a module in $\KL_{k_{n+1}}(\g_{n+1})$. Then any simple composition factor of $H_{\theta^{(n+1)}} (L_{k_{n+1}}(\lambda))$ is isomorphic to some $\mathcal A_{a, b}$.
 \end{itemize}
 \end{corollary}
 We prove the Lemma \ref{QHlowerbounded}.
\begin{proof}
Assume that the lower bound is attained for $i$ such that $i+ b =0$ and $b \ne 0$. So $\lambda_2 \ne 0$ if $b<0$ and $\lambda_m \neq 0$ if $b>0$.

 We consider the case $b<0$, the case $b>0$ is similar. 
 We have  $\lambda_2 = i = -b$ and $\langle \lambda , \omega_1 - \omega_{m+1} \rangle = a-b$.
 
 The minimal reduction of $L_{k_{n+1}}(\lambda)$ has exactly the correct top level as $gl(m)$--module. We have to compare conformal weights. 
 The conformal weight of the minimal reduction of $L_{k_{n+1}}(\lambda)$ is 
 \[
 \frac{\langle \lambda, \lambda + 2\rho\rangle }{m+1} - \langle \lambda, \frac{\theta}{2}\rangle ,
 \]
 and the top level of our induced module has conformal weight
 \[
 \frac{\lambda_2^2}{m} + \lambda_2  - \langle \lambda , \omega_1 - \omega_{m+1} \rangle^2 \frac{m+2}{2m}.
  \]
 These need to coincide, so their difference needs to vanish. Using that $\omega_ i \omega_j = \text{min}\{ i, j\} - \frac{ij}{m+2}$ we compute:

\begin{align*}
\langle \lambda,  2 \rho \rangle  &= \lambda_1 (m+1) + 2 \lambda_2 m + \lambda_{m+1} (m+1),\\
\langle  \lambda, \lambda \rangle  &= \lambda_1^2 \frac{m+1}{m+2} + \lambda_2^2 \frac{2m}{m+2} + \lambda_{m+1}^2 \frac{m+1}{m+2} + \\& + \lambda_1 \lambda_2 \frac{2m}{m+2} + \lambda_1 \lambda_{m+1} \frac{2}{m+2} + \lambda_2 \lambda_{m+1} \frac{4}{m+2},\\
 \frac{\langle \lambda, \lambda + 2\rho\rangle }{m+1} &=   \frac{\langle \lambda, \lambda \rangle }{m+1} + \lambda_1 + 2\lambda_2 \frac{m}{m+1} + \lambda_{m+1},\\
 \langle  \lambda,   \omega_1 - \omega_{m+1}\rangle   & =  \lambda_1 \frac{m}{m+2} + \lambda_2 \frac{m-2}{m+2} - \lambda_{m+1} \frac{m}{m+2}, \\
  \langle \lambda ,\omega_1 - \omega_{m+1}\rangle ^2 \frac{m+2}{2m} &= \lambda_1^2 \frac{m}{2(m+2)} + \lambda_2^2 \frac{(m-2)^2}{2m(m+2)} + \lambda_{m+1}^2 \frac{m}{2(m+2)} +\\& + \lambda_1 \lambda_2 \frac{m-2}{m+2} -  \lambda_1 \lambda_{m+1} \frac{m}{m+2} -  \lambda_2 \lambda_{m+1} \frac{m-2}{m+2}.
\end{align*}
 Now we have:
 \begin{equation*}
\begin{split}
& \frac{\langle \lambda, \lambda + 2\rho\rangle }{m+1} -\langle  \lambda,  \frac{\theta}{2} \rangle - \frac{\lambda_2^2}{m} - \lambda_2  + \langle \lambda , \omega_1 - \omega_{m+1} \rangle^2 \frac{m+2}{2m} =\\
= &  \frac{\lambda_1^2}{2} + \lambda_2^2 \frac{m-3}{2(m+1)} +  \frac{\lambda_{m+1}^2}{2} + \lambda_1 \lambda_2 \frac{m-1}{m+1} - \lambda_1 \lambda_{m+1}  \frac{m-1}{m+1}  - \lambda_2 \lambda_{m+1}  \frac{m-3}{m+1} +\\& + \frac{\lambda_1}{2}  + \lambda_2 \frac{m-3}{2(m+1)} + \frac{\lambda_{m+1}}{2} \\
= & \frac{1}{2} \left( \lambda_1 + \lambda_2 \frac{m-3}{m-1} - \lambda_{m+1} \frac{m-1}{m+1}\right)^2 + \lambda_2^2 \frac{2(m-3)}{(m+1)(m-1)^2} + \lambda_{m+1}^2 \frac{2m}{(m+1)^2} + \\& 
+ \lambda_1 \lambda_2 \frac{4}{m^2 - 1} + \frac{\lambda_1}{2}  + \lambda_2 \frac{m-3}{2(m+1)} + \frac{\lambda_{m+1}}{2} \\ >& 0.
\end{split}
\end{equation*}
Therefore, the equality of conformal weights is not possible. This proves that $i=0$, and therefore, $\lambda_2 =0$.

We turn to the typical case. The lower-bound is attained at $i=0$ and $\nu = \mu \frac{(m+2)}{m} + N$ for $N \in \{ 0, -1\}$. Consider the quantum Hamiltonian reduction of $L_{k_{n+1}}(\lambda)$ and assume that its top level coincides with the one of $\mathcal T_{\mu, \nu}$. 
  Since 
\bea   \mu = \langle \lambda, \omega_1 - \omega_{m+1}\rangle &=& \langle \lambda_1 \omega_1 + \lambda_{m+1} \omega_{m+1}, \omega_1 - \omega_{m+1}\rangle \nonumber \\
&=&  \lambda_1 \left(  (1- \tfrac{1}{m+2}) - 1 +   \tfrac{m+1}{m+2})     \right) + \lambda_{m+1}   \left(  (1- \tfrac{m+1}{m+2})  - (m+1) + \tfrac{(m+1)^2}{m+2})  \right) \nonumber  \\
&=& \lambda_1 \tfrac{m}{m+2}  - \lambda_{m+1} \tfrac{m}{m+2}   = ( \lambda_1 - \lambda_{m+1})  \tfrac{m}{m+2}, \nonumber
\eea
 we have that
 $ \nu =  \lambda \tfrac{(m+2)}{m} +N = (\lambda_1 -\lambda_{m+1})   + N \in {\Z}$  is an integer.  The proof follows.
\end{proof}

\begin{lemma}\label{KLcat} Assume that $\KL_{k_n}(\g_n)$ is a semi simple rigid braided tensor category. Then  $\KL_{k_{n+1}}(\g_{n+1})$ is a  braided tensor category.
\end{lemma}
\begin{proof} 
Let $\theta = \theta ^{(n+1)}$.
We use \cite[Theorem 3.4.3]{CY}. It sufficies to   prove  that  every  highest weight module  $W$ in $\KL_{k_{n+1}}(\g_{n+1})$ is of  finite length.  Let $\overline W$ be any proper highest weight submodule of $W$.
Let 
   $M:=H_{\theta} (W)$.  
   
 First we notice from Corollary \ref{cor:Aab} that 
 $$ M = Ind  (U _0 ^{(n)}  \mathcal  \otimes F^{\ell} _a \otimes \widetilde{\mathcal M}_b  ), $$
 where $ \widetilde{\mathcal M}_b$ is a highest weight $\mathcal M$--module that has $\mathcal M_b$ as composition factor and the conformal weight of the top level of 
 $\mathcal M_b$ is minimal among those of all composition factors.
 The following holds:
 $$ b \in \{ \frac{m+2}{m} a, \frac{m+2}{m} a \pm 1\}. $$
 
Let  $N = H_{\theta} (\overline W)$. Then $N \subsetneqq M$ and we have 
$$ N = Ind  (U _0 ^{(n)}  \mathcal  \otimes F^{\ell} _{\bar a} \otimes \widetilde{\mathcal M}_{\bar b}  ), $$
where   $$\bar  b \in \{ \frac{m+2}{m} \bar a, \frac{m+2}{m} \bar a \pm 1\}. $$
Since $  F^{\ell} _{\bar a} \otimes \widetilde{\mathcal M}_{\bar b}= Com (U _0 ^{(n)} , N)$ is a submodule  of  $F^{\ell} _{ a} \otimes \widetilde{\mathcal M}_{b}= Com (U _0 ^{(n)} , M)$, we conclude that $\bar a =  a$. 
Therefore the submodule $N$ has the form
$$ N = Ind  (U _0 ^{(n)}  \mathcal  \otimes F^{\ell} _{a} \otimes \widetilde{\mathcal M}_{\bar b}  ), $$
such that  $$\bar  b \in \{ \frac{m+2}{m} a, \frac{m+2}{m}  a \pm 1\}. $$
Since there are only finitely many possibilities for $\overline b$,     we conclude that $M$ has only finitely many highest weight submodule which come from the  reduction from highest weight submodules of $W$. This proves that $W$ has  finite length.
\end{proof}
By Theorem \ref{thm:robert} we get
\begin{corollary}\label{cor:sc}
The $U_i^{(n)}, i \in {\Z}$ are simple currents, in particular rigid, with fusion rules
$$ U_i ^{(n)} \times U_j ^{(n)} = U_{i+j}^{(n)}, \quad i, j \in {\Z}.$$
\end{corollary}

\begin{lemma} \label{nije+-1} Let  $m=2n$ and  $b = \frac{m+2}{m} a \pm 1$.  Assume that $\KL_{k_{n+1}}(\g_{n+1})$ is a braided tensor category.  Then the induced $W_{k_{n+1}}(\g_{n+1}, \theta)$--module $Ind  (U _0 ^{(n)}  \mathcal  \otimes F^{\ell} _a \otimes \widetilde{\mathcal M}_b  )$   does not have the form $H_{\theta}(W)$, where $W$ is a highest weight module from  $\KL_{k_{n+1}}(\g_{n+1})$.
\end{lemma}
\begin{proof}
Let $\lambda = \lambda_1 \omega_1 + \lambda_{m+1} \omega_{m+1}$ and assume that $W =  \widetilde  L_{k_{n+1}}(\lambda)$  is any highest weight module in $\KL_{k_{n+1}}(\g_{n+1})$ with highest weight $\lambda$.
The conformal weight of the minimal reduction of  $ \widetilde  L_{k_{n+1}}(\lambda)$ is 
\[
\Delta_\theta := \frac{1}{m+2} \lambda_1^2 + \frac{1}{m+2} \lambda_{m+1}^2 + \frac{2}{(m+1)(m+2)} \lambda_1 \lambda_{m+1} + \frac{1}{2} \lambda_1 + \frac{1}{2} \lambda_{m+1}.
\]
We consider the case $b = \frac{m+2}{m} a + 1 = \lambda_1 - \lambda_{m+1} + 1$.
Then the top level of $Ind  (U _0 ^{(n)}  \mathcal  \otimes F^{\ell} _a \otimes \widetilde{\mathcal M}_b  )$ has conformal weight
\[
\Delta \left( a,b,0 \right) = \frac{1}{m+2} b^2 + \frac{|b|}{2} + \frac{m}{m+2} b - \frac{m}{2(m+2)}.
\]
Let $b \geq 0 $. In this case we have 
\[
\Delta_\theta - \Delta \left( a,b,0 \right) = \frac{2}{m+1} \lambda_1 \lambda_{m+1} - \lambda_1 + 2 \lambda_{m+1} -1 ,
\]
so $\Delta_\theta$ and $\Delta \left( a,b,0 \right)$ coincide if and only if
	\begin{equation}\label{eq-1}
	\lambda_1 = -(m+1) - \frac{m^2 + m}{2 \lambda_{m+1} -m -1}.
\end{equation}
Since $\KL_{k_{n+1}}(\g_{n+1})$ is a braided tensor category, it follows that $L_{k_{n+1}}(s \omega_1)$ and $L_{k_{n+1}}(s \omega_{m+1})$, for $s \in \Z_{\geq 0} $, are simple-current modules (Corollary \ref{cor:sc}). Therefore, there is a non-trivial intertwining operator of the type ${ M \choose \widetilde  L_{k_{n+1}}(\lambda) \ \ L_{k_{n+1}}(\omega_1) }$, where $M$ is a certain module from $\KL_{k_{n+1}}(\g_{n+1})$. This implies that the top component of $M$ is a non-trivial summand in the tensor product of $\g_{n+1}$--modules $V_{\g_{n+1}} (\lambda) \otimes V_{\g_{n+1}}(\omega_1)$.

We have 
	\begin{equation*}
	\begin{split}
		 V_{\g_{n+1}}(\lambda_1 \omega_1 + \lambda_{m+1} \omega_{m+1}) \otimes V_{\g_{n+1}}(\omega_1) &=V_{\g_{n+1}}((\lambda_1 + 1) \omega_1 + \lambda_{m+1} \omega_{m+1}) \\& \oplus  V_{\g_{n+1}}(\lambda_1 \omega_1 + (\lambda_{m+1}-1) \omega_{m+1})   \\& \oplus V_{\g_{n+1}}((\lambda_1 - 1) \omega_1 + \omega_2 + \lambda_{m+1} \omega_{m+1}).
	\end{split}
\end{equation*} 
From (\ref{eq-1}) we conclude that none of the $\g_{n+1}$--modules that appear on the right side is a top component of the module in $\KL_{k_{n+1}}(\g_{n+1})$. Therefore, we get $L_{k_{n+1}}(\lambda) \boxtimes L_{k_{n+1}}(\omega_1) = 0$, which is a contradiction to
\begin{equation}\nonumber
\begin{split}
L_{k_{n+1}}(\lambda) &\cong L_{k_{n+1}}(\lambda) \boxtimes L_{k_{n+1}}(\g_{n+1}) \cong L_{k_{n+1}}(\g_{n+1}) \boxtimes (L_{k_{n+1}}(\omega_1) \boxtimes L_{k_{n+1}}(\omega_{m+1}))\\ &\cong 
(L_{k_{n+1}}(\lambda) \boxtimes (L_{k_{n+1}}(\omega_1)) \boxtimes L_{k_{n+1}}(\omega_{m+1}) = 0.
\end{split}
\end{equation}
Let $b < 0$. In this case we have
\[
\Delta_\theta - \Delta \left( a,b,0 \right) = \frac{2}{m+1} \lambda_1 \lambda_{m+1} +  \lambda_{m+1} ,
\]
so $\Delta_\theta$ and $\Delta \left( a,b,0 \right)$ coincide if and only if $\lambda_{m+1} = 0$. But then $b=\lambda_1 + 1 > 0$, which is a contradiction.

Analogous arguments show that the case $b = \frac{m+2}{m} a - 1$ also leads to a contradiction.
\end{proof}

Now we consider the remaining case  $b = a \frac{m+2}{m}$.  In this case the top level of the induced $W_{k_{n+1}}(\g_{n+1}, \theta)$--module $Ind  (U _0 ^{(n)}  \mathcal  \otimes F^{\ell} _a \otimes \widetilde{\mathcal M}_b  )$ has  conformal weight 
\[
\frac{b^2}{2} + \frac{|b|}{2} - \frac{m}{2(m+2)} b^2 = \frac{b^2}{m+2} + \frac{|b|}{2},
\]
which is exactly the conformal weight of the top level of the reduction of $L_{k_{n+1}}(-b \omega_1)$ (resp.   $L_{k_{n+1}}(b \omega_{m+1})$) for $b < 0$ (resp. $b > 0$).


 \begin{theorem}\label{induction step} Assume that $W$ is an highest weight module from $\KL_{k_{n+1}}(\g_{n+1})$. Then $W$ is irreducible and it has the form $W= L_{k_{n+1}}(s \omega_1)$ or $L_{k_{n+1}}(s \omega_{m+1})$ for $s \in {\Z}_{\ge 0}$. 
 Moreover, the category $\KL_{k_{n+1} }(\g_{n+1})$ is a semisimple, rigid braided tensor category.
 \end{theorem}
 \begin{proof}  
 The proof is by induction for $n \ge 4$. The case $m=4$ is proved in \cite{APV}. Assume now that the claim holds for $m=2n$. Then $W_{k_{n+1}}(\g_{n+1}, \theta)$ is the braided  tensor category and we can use Lemma \ref{nije+-1}.
 
 From the above analysis we see that $W$ has the weight $s \omega_1$ or $s\omega_{m+1}$, where $s \in {\Z}_{\ge 0}$. Let us assume that $W$ is a highest weight module with highest weight  $s\omega_1$. Assume that $W$ is not irreducible. Then it has a a proper highest weight submodule   $\overline W \subsetneqq W$. Then $H_{\theta}( \overline W)$ is a proper and non-zero submodule of $H_{\theta}(W)$.  Therefore we have
 $$ H_{\theta}(W) = Ind  (U _0 ^{(n)}  \mathcal  \otimes F^{\ell} _a \otimes  \widetilde{\mathcal   M}_b  ),  H_{\theta}( \overline W) =Ind  (U _0 ^{(n)}  \mathcal  \otimes F^{\ell} _{\overline a} \otimes  \widetilde {\mathcal M}_{\overline b}  ) \subset  H_{\theta}(W).  $$
 Using Lemma  \ref{nije+-1} we have $ b  =\frac{m+2}{m} a$ and   $\overline b  =\frac{m+2}{m} \overline a$.  Note that
  $$ \mbox{Com} ( U _0 ^{(n)}, H_{\theta}(W) ) = F^{\ell} _a \otimes  \widetilde{\mathcal   M}_b, \quad   \mbox{Com} ( U _0 ^{(n)}, H_{\theta}(\overline W) )  =   F^{\ell} _{\overline a} \otimes  \widetilde {\mathcal M}_{\overline b} ,  $$
which implies that 
$$ F^{\ell} _{\overline a} \otimes  \widetilde {\mathcal M}_{\overline b} \subset  F^{\ell} _a \otimes  \widetilde{\mathcal   M}_b, $$
and we conclude that $a = \overline a$. This  gives that $\overline b  =\frac{m+2}{m} \overline a = \frac{m+2}{m}  a = b$. This implies that  $ H_{\theta}( \overline W)  = H_{\theta}(  W)$, and therefore $H_{\theta} ( W / \overline W) = 0. $ This is a contradiction.


 Thus $W$ can not have highest weight submodules which implies that  $W$ is irreducible.
Therefore each highest weigh module $W$ in $\KL_{k_{n+1}}(\g_{n+1})$ must be  irreducible. Now the   main result of \cite{AKMPP-20}   implies that $\KL_{k_{n+1}}(\g_{n+1})$ is semi-simple. The fusion rules and existence of tensor category follows as  in  \cite{APV} and \cite{CY}.
  \end{proof}
This finishes the proof of Theorem \ref{main}.  
 

   \subsection{Ordinary modules for $W_{k_{n+1}}(\g_{n+1}, \theta^{(n+1)})$}\label{Word}
   As an application  of our results in this section above, we 
 can  describe the category of ordinary, irreducible $W_k(\g_{n+1}, \theta^{(n+1)})$--modules.
   
   \begin{theorem}Assume that $M$ is an irreducible  ordinary $W_{k_{n+1}}(\g_{n+1}, \theta^{(n+1)})$--module. Then $M$ is isomorphic to exactly one of the following modules
   \begin{itemize}
   \item The  $ \mathcal A_{a,b} = Ind  (U _0 ^{(n)}  \mathcal  \otimes F^{\ell} _{a} \otimes \mathcal M_{ b}  )$, with  $\frac{m+2}{m} a \in {\Z}$,  $b =  \frac{m+2}{m} a$ or $
   b= \frac{m+2}{m} a+ 1 >0$ or   $b= \frac{m+2}{m} a-  1 < 0$.
       \item The  $ \mathcal A^{(1)} _{a,b} = Ind  (U _{-b} ^{(n)}  \mathcal  \otimes F^{\ell} _{a} \otimes \mathcal M_{ b}  )$, with   $\frac{m+2}{m} a \in {\Z}$, $b  =   \frac{m+2}{m} a +1 <0$.
       \item The  $ \mathcal A^{(-1)} _{a,b} = Ind  (U _{-b} ^{(n)}  \mathcal  \otimes F^{\ell} _{a} \otimes \mathcal M_{ b}  )$, with   $\frac{m+2}{m} a \in {\Z}$, $b  =   \frac{m+2}{m} a -1  > 0$.
   
    \item The  $ \mathcal T  _{\mu, \nu} = Ind  (U _0 ^{(n)}  \mathcal  \otimes F^{\ell} _{\mu } \otimes \mathcal V_{ \nu}  )$, with  $  \nu =\frac{m+2}{m} \mu $ or   $\nu =\frac{m+2}{m} \mu-1$.
       \end{itemize}
       The category of ordinary $W_{k_{n+1}}(\g_{n+1}, \theta^{(n+1)})$--modules is not semi-simple.
   \end{theorem}
   \begin{proof}
   By using Zhu's algebra theory and previous results  we see that $M_{top} \cong U \otimes  \mathcal (F^{\ell} _{\mu} )_{top} \otimes (L _{\nu})_{top}$  where
   \begin{itemize}
   \item $U$ is  a highest weight, finite-dimensional,  irreducible ${\g_n}$--module which  must be a top component of an irreducible $L_{k_n} (\g_n)$--module in $\KL_{k_n} (\g_n)$.
   \item $(F^{\ell} _{\mu} )_{top}$ is top component of  irreducible $\mathcal H$--module on which $h(0)$ acts as multiplication with $\mu \in {\C}$.
   \item $\nu  \in {\C}$ and   $( L _{\nu} )_{top}$ is a top component of  highest weight   $\mathcal M(2)$--module   such that $ L _{\nu}  = \mathcal M_{\nu}$ if ${\nu}  \in {\Z}$ or $\mathcal L _{\nu} = \mathcal V_{\nu}$ otherwise.
   \end{itemize} 
   Since $ M = W_{k_{n+1}}(\g_{n+1}, \theta^{(n+1)}). M_{top}$, we conclude that $M$ is isomorphic to the induced module
  $$
M \cong  Ind  (U_{i} ^{(n)} \otimes \mathcal F^\ell_{\mu  } \otimes \mathcal  \mathcal L_{\nu}). $$
   We only need to check which induced modules are local and lower bounded, and what is lower bound. But this is discussed also in previous part of this section, so we conclude that  we have two possibilities:
   \begin{itemize}
   \item $\nu  = b \in {\Z}$ (atypical case), then  we set $a =\mu$ and  one gets that  the induced module is local if and only if $a \tfrac{m+2}{m} \in \Z$, and   lower bounded if and only if $\vert b - a \tfrac{m+2}{m} \vert \le \tfrac{3}{2}$.   Using Lemma \ref{lower-bound} we get that the lower bound is attained either for $i =0$    when
  $ b =  a \tfrac{m+2}{m} $, $b =  a \tfrac{m+2}{m} + 1 >0$,  $b =  a \tfrac{m+2}{m} - 1 <0$.
   In these cases $M \cong \mathcal A_{a,b}$. If  $b =  a \tfrac{m+2}{m} + 1 <0$, we get that $M \cong  \mathcal A^{(1)}_{a,b}$, and if $b =  a \tfrac{m+2}{m} - 1 >0$, we get that
   $M \cong  \mathcal A^{(-1)}_{a,b}$.
   \item $ \lambda \in {\C} \setminus {\Z}$ (typical case), then  the induced module is local if and only if  $\nu \equiv \frac{m+2}{m} \mu \ \mbox{mod} \  {\Z}$,  and the lower bounded if and only if $|  \nu - \frac{m+2}{m} \mu + \tfrac{1}{2}    | \leq 1$.  Using Lemma \ref{lowerbound-tipical} we have that the lower bound is attained for $i=0$. This shows that then $M \cong   \mathcal T _{\mu, \nu}$.
   
 \item  If $\nu \in {\Z}$, $\nu \equiv \frac{m+2}{m} \mu \ \mbox{mod} \  {\Z}$,  $|  \nu - \frac{m+2}{m} \mu + \tfrac{1}{2}    | \leq 1  $, then the induced module  $\mathcal T _{\mu, \nu}$ is indecomposable and reducible.
\end{itemize}
The proof follows.
   \end{proof}
   
   Define the following sets:
   \bea
   \mathcal S^{(0)} &=& \{ \mathcal A_{a,b} \ \vert b = \tfrac{m+2}{m} a \in {\Z} \}, \nonumber \\
   \mathcal S^{(1)} &=&   \{ \mathcal A_{a,b} \ \vert b = \tfrac{m+2}{m} a + 1 \in {\Z}_{>0} \},  \nonumber \\
   \mathcal S^{(-1)} &=&    \{ \mathcal A_{a,b} \ \vert b = \tfrac{m+2}{m} a-  1 \in {\Z}_{<0} \}, \nonumber  \\
     \mathcal S^{(typ)} &=&   \{ \mathcal T _{\mu,\nu} \ \vert  \nu \in {\C} \setminus {\Z},\  \nu  = \tfrac{m+2}{m} \mu \ \mbox{or} \  \nu  = \tfrac{m+2}{m} \mu-1\}. \nonumber  \eea
   
   We have the  following  corollary:
   \begin{corollary}  Vertex algebra $W_{k_{n+1}}(\g_{n+1}, \theta^{(n+1)})$ has:
   \begin{itemize}
   \item Infinitely many irreducible highest weight modules $M$ such that $\dim M_{top} =1$.
   In particular $\dim M_{top} =1$  if and only if
   $$  M \in  \mathcal S^{(0)} \bigcup \mathcal S^{(1)}  \bigcup \mathcal S^{(-1)} \bigcup     \mathcal S^{(typ)}.$$

       \item For each $i \in {\Z}$ there  a unique, up to equivalence,  irreducible  $W_{k_{n+1}}(\g_{n+1}, \theta^{(n+1)})$--module  $\mathcal L[i]$ such that $\dim \mathcal L[i]_{top} = \dim V_{\g_n}(i \omega_1)$ for $i \ge 1$, $\dim \mathcal L[i]_{top}  = \dim V_{\g_n}(-i \omega_{2n-1})$ for $i <0$. These modules are realized as
   $$ \mathcal L[i]= \mathcal A^{(1)} _{  -(i+1) \tfrac{m}{m+2}, -i } \ (i > 0), \  \mathcal L[i]= \mathcal A^{(-1)} _{  -(i-1) \tfrac{m}{m+2}, -i } \ (i < 0). $$
   \item  Irreducible module $M$ is a QHR of a module from $\KL_{k_{n+1}}(\g_{n+1}) $ if and only if $M \in    \mathcal S^{(0)}$.
   \end{itemize}
   \end{corollary}

\section{Second  approach to $ W_k(\mathfrak{sl}_{m+2}, \theta)$}\label{Wdecomp}

We shall now present a more direct decomposition of $ W_k(\mathfrak{sl}_{m+2}, \theta)$ as a $  L_{k+1} (\mathfrak{sl}_m) \otimes \mathcal H \otimes \mathcal M(2)$--module. The advantage of this second approach is that it also works for $m$ odd.

 \subsection{The $\beta \gamma$ vertex algebra $\mathcal S $}
 
   Let $\mathcal S$ be the $\beta \gamma$ vertex algebra (also called the Weyl vertex algebra) generated by fields $\beta(z) = \sum_{n \in {\Z} } \beta(n) z^{-n-1}$, $\gamma (z) = \sum_{n \in {\Z} } \gamma (n) z^{-n}$. 
  Let $J^{\mathcal S}:= :\beta \gamma:$ be the Heisenberg field. It generates the Heisenberg vertex algebra $\mathcal H$.  It is well-known (cf. \cite{KR, Wa}) that  $\mathcal S $ is a completely reducible $\mathcal H \otimes \mathcal M$--module and we have:
 $$\mathcal S = \bigoplus _{i \in {\Z} } \mathcal F^{-1} _i \otimes \mathcal M_i.$$

\subsection{The structure of  $W_k(\mathfrak{sl}_{m+2}, \theta)$}

From Theorem \ref{dec-akmpp}  we have that  $ L_{- \frac{m+1}{2}} (\mathfrak{sl}_m) \otimes \mathcal H$  is conformally embedded into  simple affine vertex algebra $L_{- \frac{m+1}{2}}(\mathfrak{sl}_{m+1})$.    Denote the generator of the Heisenberg subalgebra  inside of $L_{-  \frac{m+1}{2}}(\mathfrak{sl}_{m+1})$ by $J^{(1)}$.
We have the following decomposition:

$$
L_{-  \frac{m+1}{2}}(\mathfrak{sl}_{m+1}) = \bigoplus_{i \in \mathbb Z} U_i^{(m/2)} \otimes \mathcal F_i^s.
$$
 with $s = - \frac{m}{2}$. 
 Define
 $$ h = J^{\mathcal S} - J^{(1)}, \quad \overline h = \frac{ mJ^{\mathcal S} +2 J^{(1)}}{m+2}. $$
 Then $h$ and  $\overline h$  generate the rank one Heisenberg vertex algebras $\mathcal H_1$ and  $\mathcal H_2$, respectively, so that $\mathcal H_1 \otimes \mathcal H_2$ is a subalgebra of  $L_{-  \frac{m+1}{2}}(\mathfrak{sl}_{m+1}) \otimes  \mathcal S$.
  
  \begin{theorem}  \label{w-decomp-nova} Let  $m \in {\Z}_{\ge 0}$, $m \ge 4$. Then 
  $$ W_{ - \frac{m+3}{2}} (\mathfrak{sl}_{m+2}, \theta) \cong \mbox{Com} (\mathcal H_1, L_{- \frac{m+1}{2}}(\mathfrak{sl}_{m+1}) \otimes  \mathcal S ) $$
  and we have the following decomposition 
  $$ W_{ - \frac{m+3}{2}} (\mathfrak{sl}_{m+2}, \theta)  \cong \bigoplus_{i \in {\Z}} U_i ^{(m/2)} \otimes \mathcal F_i ^{\ell} \otimes \mathcal M_i, $$
  with $\ell = - \frac{m}{m+2}$. 
  \end{theorem}
\begin{proof}

First we notice that $\mathcal H_1 \otimes \mathcal H_2 \cong \mathcal H \otimes \mathcal H^{\mathcal S}$. Then we have
\bea \mathcal F_i ^{-1}  \otimes \mathcal F_j  ^{s}  \cong \mathcal F^{s_1} _{i-j} \otimes F^{\ell}_{\frac{m i + 2 j}{m+2}},  \label{formula-vazna} \eea
with $s_1 = -\frac{m+2}{2}$, and $\ell =  - \frac{m}{m+2}$. 

From the decompositions of $\mathcal S$ and $L_{- \frac{m+1}{2}}(\mathfrak{sl}_{m+1})$ we get that as a $ \mathcal H \otimes \mathcal H^{\mathcal S}$--module:
$$ L_{-\frac{m+1}{2}} (\mathfrak{sl}_{m+1})  \otimes \mathcal S \cong \bigoplus_{i,j \in {\Z}} U_j ^{(m/2)} \otimes \mathcal F_i ^{-1}  \otimes \mathcal F_j  ^{s} \otimes \mathcal M_i. $$
Now (\ref{formula-vazna}) implies that
$$ A =  \mbox{Com} (\mathcal H_1, L_{- \frac{m+1}{2}}(\mathfrak{sl}_{m+1}) \otimes  \mathcal S )  \cong \bigoplus_{i \in {\Z}} U_i ^{(m/2)} \otimes \mathcal F_i ^{\ell} \otimes \mathcal M_i. $$
One can easily show that $A$ is a simple vertex algebra, generated by $$L_{-\frac{m+1}{2}}(\mathfrak{sl}_m) \otimes \mathcal H_2 \otimes \mathcal M(2) + U_1 ^{(m/2)} \otimes \mathcal F_1 ^{\ell} \otimes \mathcal M_1 +  U_{-1} ^{(m/2)} \otimes \mathcal F_{-1} ^{\ell} \otimes \mathcal M_{-1}, $$ and all assumptions of   \cite[Theorem 3.2]{ACKL}  are satisfied. Using this uniqueness result, we get that $A = W_{ -\frac{m+3}{2}} (\mathfrak{sl}_{m+2}, \theta)$. The proof follows.
\end{proof}

\begin{remark} One can  show that a suitable modified version  of  Theorem  \ref{w-decomp-nova} also  holds also for $m=2,3$. In this cases the algebra $U_{0} ^{(m/2)}$ can be defined is certain (infinite) extensions of $L_{-\frac{m+1}{2}} (\mathfrak{sl}_m)$. In particular, the vertex algebra $U^{(1)} _0$ is isomorphic to the vertex algebra $ \mathcal V^{(2)} _0$ already investigated in \cite{  A-TG,  ACGY, AMW, AW}.
\end{remark}

\end{document}